\documentclass[12pt,reqno]{amsart}
\usepackage{hyperref}
\usepackage{amsmath,amssymb,amsfonts}
\usepackage[english]{babel}
\usepackage{caption}
\usepackage{amssymb,amsmath,euscript,enumerate,tikz}
\usepackage[margin=1in]{geometry}
\usepackage{graphicx}
\usepackage{float}
\usepackage{pgf,tikz}
\usepackage{mathrsfs}
\usepackage{upgreek}

\newcommand \reg{\operatorname{reg}}

\newcommand \ini{\operatorname{in}}

\newcommand \HS{\operatorname{HS}}
\newcommand \R{\mathcal{R}}
\newcommand \F{\mathcal{F}}

\newcommand \K{\mathbb{K}}

\newcommand{\KK}{\mathbb{K}}

\newtheorem{theorem}{Theorem}[section]

\newtheorem{corollary}[theorem]{Corollary}

\newtheorem{notation}[theorem]{Notation}

\begin{document}
	\title[Rees algebra and special fiber ring of binomial edge ideals of closed graphs]{Rees algebra and special fiber ring of binomial edge ideals of closed graphs}
	\author[Arvind Kumar]{Arvind Kumar}
	\email{arvindkumar@cmi.ac.in}
	\address{Department of Mathematics, Chennai Mathematical Institute, Siruseri
		Kelambakkam, Chennai, India - 603103}

	\begin{abstract} In this article, we compute the regularity of Rees algebra of binomial edge ideals of closed graphs. We obtain a lower bound for the regularity of Rees algebra of binomial edge ideals. We also study some algebraic properties of the Rees algebra and special fiber ring of binomial edge ideals of closed graphs via algebraic properties of their initial algebra and Sagbi basis theory. We obtain an upper bound for the regularity of the special fiber ring of binomial edge ideals of closed graphs. 
	\end{abstract}
	\maketitle
	\section{Introduction}
	Let $S$ be a standard graded polynomial ring over a field $\KK$. Let $I$ be a homogenous ideal in $S$. The $\KK$-subalgebra  $\displaystyle\mathop\oplus_{k\geq0}I^kt^k \subset S[t]$ is known as the  {\it Rees algebra} of $I$, and  is denoted by $\mathop\mathcal{R}(I)$.  The Rees algebra of a homogeneous ideal encodes a lot of asymptotic properties of that ideal. In this paper, we study the Rees algebra of binomial edge ideals of closed graphs.  
	
	An ideal generated by a set of $2$-minors of a $2 \times n$ generic matrix is known as {\it binomial edge ideal}.  These ideals were
	introduced by Herzog et al. in \cite{HH1} and independently by Ohtani in \cite{oh} a decade ago.   Let $G$ be a  simple graph with vertex set  $V(G) =[n]:= \{1, \ldots, n\}$ and edge set $E(G)$.  The binomial edge ideal of $G$ is defined  as $J_G = (
	x_i y_j - x_j y_i ~ : i < j \text{ and } \{i,j\} \in E(G)) \subset S=
	\K[x_1, \ldots, x_{n}, y_1, \ldots, y_{n}]$. Since binomial edge ideals are in one-to-one correspondence with finite simple graphs, many authors have intensively studied these ideals' algebraic properties in terms of combinatorial properties of graphs. For example, Gr\"obner basis and primary decomposition of binomial edge ideals have been computed in terms of combinatorial invariants of graphs in \cite{MG, HH1, oh}, Castelnuovo-Mumford regularity of binomial edge ideals has been studied in terms of various combinatorial invariants of graphs in \cite{JACM, AR3, MM, MMK1}, extremal Betti numbers of block graphs and generalized block graphs have been studied in terms of minimal cut sets of graphs in \cite{her2, AR2}, and Cohen-Macaulayness of binomial edge ideals has been studied in \cite{dav, dav2, her1, Rauf}. 
	
	While the study of binomial edge ideals has been explored by many researchers, much less is known about their Rees algebra. The study of the Rees algebra of binomial edge ideals has been initiated in \cite{JAR}.  The authors in \cite{JAR} obtained the defining ideal of Rees algebra of almost complete intersection binomial edge ideals and prove that the Rees algebra of almost complete intersection binomial edge ideals are Cohen-Macaulay.  Recently, the authors in \cite{ERT} showed that the Rees algebra of binomial edge ideals of closed graphs is Cohen-Macaulay. In this paper, we prove that the regularity of Rees algebra of binomial edge ideal of a graph is bounded below by the length of that graph's longest induced path. We obtain the regularity of Rees algebra of binomial edge ideals of closed graphs. We then study some algebraic properties of the Rees algebra of binomial edge ideals of closed graphs via algebraic properties of its initial algebra and Sagbi basis theory.
	
	Another $\KK$-algebra associated with a  homogeneous ideal is the special fiber ring. The {\it special fiber ring } of a homogeneous ideal $I$  is the ring $\displaystyle \mathcal{F}(I)= \R(I)/\mathfrak{m}\R(I) \cong \mathop\oplus_{k \geq 0} I^k/\mathfrak{m}I^{k}$, where $\mathfrak{m}$ is the homogeneous  maximal ideal of $S$. Nothing much is known about the special fiber ring of binomial edge ideals. The author in \cite{AR4} proved that the special fiber ring of binomial edge ideals of forests is a polynomial ring.  In this paper, we study the special fiber ring of binomial edge ideals of closed graphs. We prove that the special fiber ring of binomial edge ideals of closed graphs is Koszul and  Cohen-Macaulay normal. We obtain the analytic spread of binomial edge ideals of closed graphs in terms of the number of vertices and the number of indecomposable components of $G$. We also obtain an upper bound for the regularity of the special fiber ring of binomial edge ideals of closed graphs. 
	
	The article is organized as follows. The second section contains all
	the necessary definitions and notation required in the rest of the article. In Section 3, we study the  Rees algebra of binomial edge ideals of closed graphs.  We study the special fiber ring of binomial edge ideals of closed graphs in Section 4.
	
	\section{Preliminaries}
	In this section, we collect all the notions that we use in this paper. We first recall all the necessary definitions from graph theory.

	Let $G$  be a  simple graph with the vertex set $[n]$ and edge set
	$E(G)$. A graph on $[n]$ is said to be a \textit{complete graph}, if
	$\{i,j\} \in E(G)$ for all $1 \leq i < j \leq n$. The complete graph on
	$[n]$ is denoted by $K_n$. For $A \subseteq V(G)$, $G[A]$ denotes the
	\textit{induced subgraph} of $G$ on the vertex set $A$, that is, for
	$i, j \in A$, $\{i,j\} \in E(G[A])$ if and only if $ \{i,j\} \in
	E(G)$.   A subset
	$U$ of $V(G)$ is said to be a \textit{clique} if $G[U]$ is a complete
	graph. A vertex is said to be a {\it simplicial vertex} if it belongs to exactly one maximal clique. The \textit{clique number}
	of a graph ${G}$, denoted by $\omega({G})$, is the maximum size of the maximal cliques of ${G}$.  For a vertex $v$,
	$N_G(v) = \{u \in V(G) :  \{u,v\} \in E(G)\}$ denotes the
	\textit{neighborhood} of $v$ in $G$. 
	The \textit{degree} of a vertex  $v$, denoted by $\deg_G(v)$, is
	$|N_G(v)|$.    A
	\textit{cycle} is a connected graph $G$ with $\deg_G(v) = 2$ for all $v \in V(G)$. A connected graph is a \textit{tree} if it does not have a
	cycle.  A graph is said to be a {\it forest} if each connected component is a tree. A graph $G$ is said to be \textit{bipartite} if 
	there is a bipartition of $V(G)=V_1 \sqcup V_2$ such that for each
	$i=1,2$, no two of the vertices of $V_i$ are adjacent. A graph is  called a  \textit{non-bipartite} graph if it is not a bipartite graph. A subset $M$ of $E(G)$ is  said to be a {\it matching} of $G$ if $ e \cap e' =\emptyset$ for every pair $e,e' \in M$ with $e \neq e'$. The {\it matching
		number} of a graph $G$, denoted by $\text{mat}(G)$, is the maximum  size of the maximal matchings of $G$.  A matching $M$ is said to be a {\it perfect matching} if $\displaystyle V(G) = \cup_{e \in M} e$. A graph $G$ is said  to be $H$-free if $H$ is not an induced subgraph of $G$. 
	
	We recall the notation of decomposability from \cite{Rauf}. 
	A graph $G$ is called \textit{decomposable}, if there exist subgraphs $G_1$ and $ G_2$ such that $G$ is obtained by identifying 
	a simplicial vertex $v_1$ of $G_1$ with a simplicial vertex $v_2$ of $G_2 $, i.e., 
	$G= G_1 \cup G_2$ with $V(G_1)\cap V(G_2)=\{v\}$ such that $v$ is a 
	simplicial vertex of both $G_1$ and $G_2$. A graph $G$ is called \textit{indecomposable}, if it is not decomposable. Up to ordering, $G$
	has  a unique decomposition into indecomposable subgraphs, i.e., there exist 
	$G_1,\ldots,G_r$ indecomposable induced subgraphs of $G$ with 
	$G=G_1\cup \cdots \cup G_r$ such that for each $i \neq j$, either 
	$V(G_i) \cap V(G_j) = \emptyset$ or $V(G_i) \cap V(G_j) =\{v\}$ and $v$ is a  simplicial vertex 
	of both $G_i$ and $G_j$. 
	
	Now, we recall all the necessary notation from commutative algebra.
	Let $R = \KK[x_1,\ldots,x_m]$ be a standard graded polynomial ring over an arbitrary
	field $\KK$ and $M$ be a finitely generated graded  $R$-module. 
	Let
	\[
	0 \longrightarrow \bigoplus_{j \in \mathbb{Z}} R(-j)^{\beta_{p,j}^R(M)} 
	\overset{\phi_{p}}{\longrightarrow} \cdots \overset{\phi_1}{\longrightarrow} \bigoplus_{j \in \mathbb{Z}} R(-j)^{\beta_{0,j}^R(M)} 
	\overset{\phi_0}{\longrightarrow} M\longrightarrow 0,
	\]
	be the minimal graded free resolution of $M$, where
	$R(-j)$ is the free $R$-module of rank $1$ generated in degree $j$.
	The number $\beta_{i,j}^R(M)$ is  called the 
	{\it $(i,j)$-th graded Betti number of $M$}. The 
	{\it Castelnuovo-Mumford regularity (simply regularity) of $M$}, denoted by $\reg(M)$, is defined as 
	\[
	\reg(M):=\max \{j-i : \beta_{i,j}^R(M) \neq 0\}
	.\]
	
	Let $G$ be a graph on $[n]$ and $J_{G}$ be its binomial edge ideal in the standard graded  polynomial ring $S=\KK[x_1,\ldots,x_n,y_1,\ldots,y_n]$. Then, $G$ is said to be {\it closed with respect to given labelling of vertices} if the generators of $J_G$ form a quadratic  Gr\"obner basis of $J_G$ with respect to the lexicographic term order on $S$ induced by $x_1 > \cdots > x_n > y_1 > \cdots > y_n$. We say that a graph is a {\it  closed graph }  if it is closed  with respect to some labelling of vertices.
	
	\section{Rees Algebra of binomial edge ideals of closed graphs}
	In this section, we study the regularity of Rees algebra of binomial edge ideals of closed graphs. We also study some other algebraic properties of $\R(J_G)$ via algebraic properties of  $\R(\ini_{\sigma}(J_G))$ and Sagbi basis theory. 
	
	Let $\KK $ be a field and $S=\KK[x_1,\ldots,x_n]$ be a standard graded polynomial ring over $\KK$. Let $A$ be a finitely generated $\KK-$subalgebra of $S$ generated by homogeneous elements. Let $\sigma $ be a term order for the monomials in $S$. The {\it initial algebra} of $A$ with respect to the term order $\sigma$ is the $\KK-$subalgebra of $S$ generated by $\{ \ini_{\sigma}(f) \; : \; f \in A\}$. Let $A_i$ denote the homogeneous component of degree $i$.  The $\KK-$vector space spanned by $\{ \ini_{\sigma}(f) \; : \; f \in A_i\}$ is denoted by $\ini_{\sigma}(A_i)$. It follows from \cite[Proposition 2.4]{CHV} that 
	\[ \ini_{\sigma}(A) =\mathop\oplus_{\substack{i \geq 0}}  \ini_{\sigma}(A_i). \] Moreover, the Hilbert functions of $A$ and $\ini_{\sigma}(A)$ concide.

	Let $I$ be a homogeneous ideal in $S$. The {\it Rees algebra } of $I$ is defined as \begin{align*}
	\R(I)	= \mathop\oplus_{i \geq 0} I^it^i.\end{align*}   
	Throughout this article, we assume that $I$ is an equi-generated homogeneous ideal. Let $ \{f_1,\ldots,f_m\}$ be a minimal homogeneous generating set of $I$, and $R=S[t_1,\ldots,t_m]$ be a standard graded polynomial ring over $\KK$.
	Let $\psi : R \to S[t]$ be the $S$-algebra homomorphism given by
	$\psi(t_{i}) = f_{i}t$. Then, $R/\ker (\psi) \simeq$  Im$(\psi) = \R(I)$, where $\ker (\psi)$ is a homogeneous ideal of $R$  and it is 
	called the \textit{defining ideal} of $\R(I)$. The {\it relation type} of an ideal $I$ is the  largest $t$-degree of a minimal generator of the defining ideal of $\R(I)$. We say that $I$ is of {\it linear type} if Rees algebra of $I$ is isomorphic to the symmetric algebra of $I$. If the defining ideal of Rees algebra is generated by defining equations of the symmetric algebra and the defining equations of the special fiber ring, then we say that $I$ is of  {\it fiber type}. 
	
	Let $\sigma $ be a term order for the monomials in $S$. Let $\tau$ be term order for monomials in  $S[t]$ defined as follows:  given two monomials $u,v\in S$ and two integers $i,j\geq 0,$ we have $ut^i <_\tau vt^j$ if and only if $i<j$ or $i=j$ and $u<_{\sigma}v.$ Then, the  {initial algebra} of $\R(I)$ with respect to the term order $\tau$ is \begin{align*} \ini_{\tau}(\R(I)) = \mathop\oplus_{i \geq 0} \ini_{\sigma} (I^i)t^i.\end{align*}
	
	\begin{notation}
		Let $I$ be an equi-generated homogeneous ideal in  $S$. Then, by $\reg(\R(I))$, we mean the regularity of $R/\ker(\psi)$ as $R$-graded module. 
	\end{notation}
	\begin{theorem}\label{main-rees-reg}
		Let $I=( f_1,\ldots,f_m) \subset S$ be an equi-generated homogeneous ideal which satisfies the followings:
		\begin{enumerate}
			\item $\ini_{\sigma}(I) = (\ini_{\sigma}(f_1),\ldots, \ini_{\sigma}(f_m) )$;
			\item $\ini_{\sigma}(I^s) =(\ini_{\sigma}(I))^s$, for all $s\geq 1$;
			\item $\R(\ini_{\sigma}(I)) $ is Cohen-Macaulay.
		\end{enumerate} Then, $\reg(\R(I)) =\reg(\R(\ini_{\sigma}(I)))$.
	\end{theorem}
	\begin{proof}
		Since  $\ini_{\sigma}(I^s) =(\ini_{\sigma}(I))^s$, for all $s\geq 1$, it follows from \cite[Theorem 2.7]{CHV} that $\ini_{\tau}(\R(I)) = \R(\ini_{\sigma}(I))$. Thus, it is enough to prove that $\reg(\R(I)) =\reg(\ini_{\tau}(\R(I)))$. Since $\ini_{\tau}(\R(I))$ is Cohen-Macaulay, by \cite[Corollary 2.3]{CHV}, $\R(I)$ is Cohen-Macaulay, and Krull dimension of $\R(I)$ and $\ini_{\tau}(\R(I))$ are same. It follows from \cite[Proposition 2.4]{CHV} that $$\HS_{\R(I)}(\lambda)  =\HS_{\ini_{\tau}(\R(I))}(\lambda),$$ and therefore, $${h}_{\R(I)}(\lambda)  ={h}_{\ini_{\tau}(\R(I))}(\lambda),$$ where $h_{\R(I)} (\lambda)$ is the $h$-polynomial of $\R(I)$. For a standard graded Cohen-Macaulay algebra $S/J$, it is well known  that $\reg(S/J)= \deg h_{S/J}(\lambda)$, (see \cite[Section 4.1]{Ene}).  Therefore, $\reg(\R(I))=\deg {h}_{\R(I)}(\lambda)$ and $\reg(\ini_{\tau}(\R(I)))= \deg {h}_{\ini_{\tau}(\R(I))}(\lambda)$. Hence, $$\reg(\R(I))=\reg(\ini_{\tau}(\R(I)))=\reg(\R(\ini_{\sigma}(I)))$$ which proves the assertion.  
	\end{proof}
	
	Let $G$ be a graph on $[n]$ and $J_{G}$ be its binomial edge ideal in the standard graded  polynomial ring $S=\KK[x_1,\ldots,x_n,y_1,\ldots,y_n]$. Let $\sigma $ be the lexicographic term order on $S$ induced by $x_1 > \cdots > x_n > y_1 > \cdots > y_n$. A graph $G$ is said to be {\it closed with respect to given labelling of vertices} if the generators of $J_G$ form a quadratic  Gr\"obner basis of $J_G$ with respect to the term order $\sigma$. We say that a graph is a {\it  closed graph }  if it is closed  with respect to some labelling of vertices.   Set $R=S[T_{\{i,j\}} ~ : ~\{i,j\} \in E(G) \text{ with } i<j]$ to be standard graded poynomial ring over $\K$.

	We now compute the regularity of Rees algebra of binomial edge ideals of closed graphs.
	\begin{theorem}\label{rees-reg}
		Let $G$ be a closed graph on $[n]$ with respect to given labelling of vertices. Assume that $G$ has no isolated vertices. Then, $$\reg(\R(J_G))=\reg(\R(\ini_{\sigma}(J_G))) =n-c,$$
		where $c$ is the number of connected components of $G$.
	\end{theorem}
	\begin{proof}
		Since $G$ is a closed graph with respect to given labelling of vertices, by \cite[Theorem 1.1]{HH1}, $\ini_{\sigma}(J_G)= (x_iy_j : i<j, \{i,j\} \in E(G))$. By \cite[Equation 3]{ERT}, $\ini_{\sigma}(J_G^s) =(\ini_{\sigma}(J_G))^s$, for all $s \geq 1$ and by \cite[Proposition 2.9]{ERT},  $\R(\ini_{\sigma}(J_G))$ is Cohen-Macaulay.   Thus, by Theorem \ref{main-rees-reg},  $\reg(\R(J_G))=\reg(\R(\ini_{\sigma}(J_G)))$. 
		
		First, we assume that $G$ is connected. Let $H$ be a graph on the vertex set $\{x_1 \ldots,x_{n-1}\} \sqcup \{y_2,\ldots, y_n\}$ and edge set $\{\{x_i,y_j\} : i <j, \{i,j\} \in E(G)\}$. It follows from \cite[Lemma 3.3]{EZ} that $H$ is a bipartite graph and the monomial edge ideal  of $H$ is $I(H)=\ini_{\sigma}(J_G)$. By \cite[Section 2]{her1}, $\{i,i+1\} \in E(G)$ for $1 \leq i \leq n-1$, therefore $\{ \{ x_i,y _{i+1}\} : 1 \leq i \leq n-1 \} \subset E(H)$ is a perfect matching of $H$.  Therefore, by \cite[Theorem  4.2]{Cid}, $\reg(\R(I(H))) =\text{mat}(H)=n-1$, where $\text{mat}(H)$ is the matching number of $H$. Thus,  $\reg(\R(J_G))=\reg(\R(\ini_{\sigma}(J_G))) =n-1.$ 
		
		Now, assume that $G$ is not connected. Let $G_1 , \ldots, G_c$ be the connected components of $G$. For each $ 1 \leq k \leq c$, let $H_k$ be the bipartite graph such that $I(H_k) = \ini_{\sigma}(J_{G_k})$. Then, $H=H_1 \sqcup \cdots \sqcup H_k$ is a bipartite graph with a perfect matching of size $n-c$. By virtue of  \cite[Theorem  4.2]{Cid}, we have $\reg(\R(I(H))) =\text{mat}(H)=n-c$. Hence,   $\reg(\R(J_G))=\reg(\R(\ini_{\sigma}(J_G))) =n-c.$ 
	\end{proof}
	\begin{corollary}\label{reg-closed}
		Let $G$ be a closed graph on  $[n]$. Assume that $G$ has no isolated vertices. Then, $\reg(\R(J_G)) =n-c$, where $c$ is the number of connected components of $G$.
	\end{corollary}
	
	Now, we obtain  a lower bound for the regularity of Rees algebra of binomial edge ideals.
	\begin{theorem}
		Let $G$ be a graph on $[n]$. If $H$ is an induced subgraph of $G$, then $$\reg(\R(J_H)) \leq \reg(\R(J_G)).$$ In particular, $\ell(G) \leq \reg(\R(J_G))$, where $\ell(G)$ is the length of a longest induced path in $G$.
	\end{theorem}
	\begin{proof}
		Let $H$ be an induced subgraph of $G$. It follows from the proof of \cite[Proposition 3.3]{JAR1} that $J_H^s = J_G^s \cap S_H$, where $S_H=\KK[x_j,y_j : j \in V(H)]$. Therefore, for every $i \geq 0$, $J_H^i t^i =J_G^i t^i \cap S_H[t]$. Then, $\displaystyle \R(J_G) \cap S_H[t] =(\mathop\oplus\limits_{i \geq 0} J_G^i t^i) \cap S_H[t]=\mathop\oplus\limits_{i \geq 0} \left(J_G^i t^i \cap S_H[t]\right) =\mathop\oplus\limits_{i \geq 0} J_H^i t^i=\R(J_H)$. Thus, $\R(J_H)$ is a $\KK$-subalgebra of $\R(J_G)$.  Set $R_H=S_H[T_{\{i,j\}} : \{i,j\} \in E(H) \text{ with } i < j]$. Let $I_1$ and $I_2$ be ideals of $R$ and $R_H$, respectively such that $R/I_1 \simeq \R(J_G)$, and $R_H/I_2 \simeq \R(J_H)$. Now, define $\pi : R/I_1 \rightarrow R_H/I_2$ as $\pi(\overline{x}_j)=\pi(\overline{y}_j)=0$ if $j \not\in V(H)$,  $\pi(\overline{x}_j) =\overline{x}_j, \pi(\overline{y}_j)=\overline{y}_j$ if $j \in V(H)$, $\pi(\overline{T}_{\{i,j\}})=0$ if $\{i,j\} \not\in E(H)$, and  $\pi(\overline{T}_{\{i,j\}})=\overline{T}_{\{i,j\}}$ if $\{i,j\} \in E(H)$.  Consider, $R_H/I_2 \stackrel{i}\hookrightarrow R/I_1 \xrightarrow{\pi} R_H/I_2$. Then, $\pi \circ i$ is identity on $R_H/I_2$, and hence, $R_H/I_2$ is an algebra retract of $R/I_1$. It follows from \cite[Corollary 2.5]{HHH} that $\reg(R_H/I_2) \leq \reg(R/I_1)$. Hence, $\reg(\R(J_H) ) \leq \reg(\R(J_G))$. 
		
		Let $H$ be a longest induced path of $G$. Then, $H$ is an induced path of $G$.  Since $H$ is a closed graph, by Corollary \ref{reg-closed}, $\reg(\R(J_H))=|V(H)|-1=\ell(G)$. Hence, $\reg(\R(J_G)) \geq \ell(G)$.
	\end{proof}
	
	We now move on to study a Sagbi basis for $\R(J_G)$, and using that we study some of algebraic properties of $\R(J_G)$ via algebraic properties of $\ini_{\tau}(\R(J_G))$. 
	\begin{theorem}\label{rees-Sagbi}
		Let $G$ be a closed graph on $[n]$ with respect to given labelling of vertices. Then, $\ini_{\tau}(\R(J_G)) = \R(\ini_{\sigma}(J_G))$ and $\text{reltype}(J_G) \leq 2$. Moreover, the set $\{x_i, y_i  : 1 \leq i \leq n\} \cup \{f_{e}t :  e \in E(G)\}$ is a Sagbi basis of $\R(J_G)$ with respect to term order $\tau$.
	\end{theorem}
	\begin{proof}
		Let $H$ be the graph constructed in the proof of Theorem \ref{rees-reg}, i.e. $I(H)=\ini_{\sigma}(J_G)$.  By  \cite[Equation 3]{ERT}, we have $\ini_{\sigma}(J_G^s) =(\ini_{\sigma}(J_G))^s$, for all $s \geq 1$. Now,  \cite[Theorem 2.7]{CHV} yields that $\ini_{\tau}(\R(J_G)) = \R(\ini_{\sigma}(J_G))$. By \cite[Lemma~3.3]{EZ}, $H$ is a bipartite graph and every induced cycle in $H$ has length $4$.    Consequently,  $\F(\ini_{\sigma}(J_G))$ is a Koszul algebra, by \cite[Theorem 1]{OH2}. Therefore, the defining ideal of $\F(\ini_{\sigma}(J_G))$ is generated in degree at most two. It follows from \cite[Theorem 3.1]{Vill95} that $\R(\ini_{\sigma}(J_G))$ is of fiber type. Thus,  $\text{reltype}(\ini_{\sigma}(J_G))\leq 2$, and hence, by \cite[Corollary 2.8]{CHV}, $\text{reltype}(J_G)\leq 2$. Since $\ini_{\tau}(\R(J_G)) = \R(\ini_{\sigma}(J_G))$ and  $\{x_i, y_i  : 1 \leq i \leq n\} \cup \{\ini_{\sigma}(f_{e})t :  e \in E(G)\}$ generate $\R(\ini_{\sigma}(J_G))$ as $\K$-algebra, the set $\{x_i, y_i  : 1 \leq i \leq n\} \cup \{f_{e}t :  e \in E(G)\}$ is a Sagbi basis of $\R(J_G)$.
	\end{proof}
	We have seen in the above theorem that binomial edge ideals of closed graphs are of quadratic type. Now, we characterize linear type closed graphs. 
	\begin{corollary}
		Let $G$ be a closed graph on $[n]$ with respect to given labelling of vertices. Then, $J_G$ is of linear type if and only if $G$ is a $K_4$-free graph. 
	\end{corollary}
	\begin{proof}
		Assume that $J_G$ is of linear type. By \cite[Proposition 5.7]{AR4}, $G$ is a $K_4$-free graph. Conversely, we assume that $G$ is a $K_4$-free graph. Let $H$ be the graph constructed in the proof of Theorem \ref{rees-reg}, i.e. $I(H)=\ini_{\sigma}(J_G)$. Then, $H$ is a forest. By \cite[Corollary 3.2]{Vill95}, $\ini_{\sigma}(J_G)$ is of linear type, and hence, by \cite[Corollary 2.8]{CHV}, $J_G$ is of linear type.
	\end{proof} 
	\begin{theorem}
		Let $G$ be a closed graph on $[n]$. If $\text{char}(\K) =0$, then $\R(J_G)$ has rational singularities, and if $\text{char}(\K)>0$, then $\R(J_G)$ is  $F$-rational. In particular, $\R(J_G)$ is a Cohen-Macaulay normal domain. 
	\end{theorem}
	\begin{proof}
		Assume, without loss of generality, that $G$ is closed with respect to given labelling of vertices. Let $H$ be the graph constructed in the proof of Theorem \ref{rees-reg}, i.e. $I(H)=\ini_{\sigma}(J_G)$. By \cite[Lemma~3.3]{EZ}, $H$ is a bipartite graph. It follows from \cite[Corollary 5.3, Theorem 5.9]{SVV} that $\R(\ini_{\sigma}(J_G))$ is a Cohen-Macaulay normal domain. Now,  Theorem \ref{rees-Sagbi} yields that $\ini_{\tau}(\R(J_G)) = \R(\ini_{\sigma}(J_G))$. Thus, $\ini_{\tau}(\R(J_G))$ is a Cohen-Macaulay normal domain. Hence, the assertion follows from \cite[Corollary 2.3]{CHV}.
	\end{proof}

	\section{Special fiber ring of binomial edge ideals of closed graphs} In this section, we study the special fiber ring of binomial edge ideals of closed graphs. We begin with definitions.  Let $\mathfrak{m}$ denote the unique homogeneous maximal ideal of $S$. The \textit{special fiber ring} of  a homogeneous ideal $I$ is the ring $\displaystyle \mathcal{F}(I)= \R(I)/\mathfrak{m}\R(I) \cong \mathop\oplus_{k \geq 0} I^k/\mathfrak{m}I^{k}$. The \textit{analytic spread} of $I$ is the Krull dimension of $\F(I)$, and it is denoted by $\ell(I)$. 
	\begin{theorem}\label{normal}
		Let $G$ be a connected closed graph on $[n]$. Then,
		\begin{enumerate}
			\item  $\F(J_G)$ is a Koszul algebra.
			\item if $\text{char}(\K)=0$, then $\F(J_G)$ has rational singularities.
			\item if $\text{char}(\K)>0$, then $\F(J_G)$ is $F$-rational.
			\item $\F(J_G)$ is a  Cohen-Macaulay normal domain.
			\item $\ell(J_G)= 2n-r-2$, where $r $ is the number of indecomposable components of $G$.
			\item $\ell(J_G)=|E(G)|$ if and only if $\omega(G) \leq 3$. 
		\end{enumerate} 
	\end{theorem}
	\begin{proof} Assume, without loss of generality, that $G$ is  closed with respect to given labelling of vertices.
		By \cite[Theorem 2.10]{ERT}, $\{f_{e} : e \in E(G)\}$ is a Sagbi basis of the $\K$-algebra $\F(J_G)$ with respect to the term order $\sigma$ on $S,$ i.e.,  \[\ini_{\sigma}(\F(J_G))=\F(\ini_{\sigma}(J_G)).\] \par (1) Let $H$ be the graph constructed in the proof of Theorem \ref{rees-reg}, i.e. $I(H)=\ini_{\sigma}(J_G)$. It follows from \cite[Lemma~3.3]{EZ} that $H$ is a bipartite graph and every induced cycle in $H$ has length $4.$ Thus, by \cite[Theorem 1]{OH2}, $\F(\ini_{\sigma}(J_G))$ is a Koszul algebra. Hence, by \cite[Corollary 2.6]{CHV}, $\F(J_G)$ is a Koszul algebra. 
		
		\par (2-4) By \cite[Corollary 1.3]{OH1},  $\F(\ini_{\sigma}(J_G))$ is normal. Now, the assertion follows from \cite[Corollary 2.3]{CHV}. 
		
		\par (5) By \cite[Proposition 2.4]{CHV}, $\ell(J_G) =\ell(\ini_{\sigma}(J_G))$. Therefore, it is enough to find $\ell(\ini_{\sigma}(J_G))$. Let $G_1, \ldots, G_r$ be the indecomposable components of $G$. Note that  $H=H_1 \sqcup \ldots \sqcup H_r$, where $H_k$ is the connected bipartite graph such that $I(H_k) =\ini_{\sigma}(J_{G_k})$. Then, $\displaystyle\F(I(H)) = \F(I(H_1)) \otimes_{\KK} \cdots \otimes_{\KK} \F(I(H_r))$. It follows from \cite[Lemma 3.1, Proposition 3.2]{Vill95} that $\displaystyle\ell(I(H)) = \sum\limits_{i=1}^r |V(H_i)| -r=\sum\limits_{i=1}^r (2|V(G_i)| -2) -r=2\sum\limits_{i=1}^r |V(G_i)| -3r=2(n+r-1)-3r=2n-r-2.$ 
		
		\par (6) Note that $H$ is a forest if and only if $\omega(G) \leq 3$. If $\omega(G) \leq 3$, then $H$ is a forest, and therefore, $\ell(J_G) =\ell(I(H))=|V(H)|-r= |E(H)| =|E(G)|$. If $ \omega(G) \geq 4$, then $H$ is not a forest, and therefore, $\ell(J_G) =\ell(I(H))=|V(H)|-r < |E(H)| =|E(G)|$. Hence, the assertion follows.
	\end{proof}
	As an immediate consequence, we obtain the following:
	\begin{corollary}\label{fiber-cor}
		Let $G=G_1\sqcup \cdots \sqcup G_c$ be a closed graph on $[n]$. Then, 
		\begin{enumerate}
			\item  $\F(J_G)$ is a Koszul algebra.
			\item if $\text{char}(\K)=0$, then $\F(J_G)$ has rational singularities.
			\item if $\text{char}(\K)>0$, then $\F(J_G)$ is $F$-rational.
			\item $\F(J_G)$ is a  Cohen-Macaulay normal domain.
			\item $\ell(J_G)= 2n-r-2c$, where $r $ is the number of indecomposable components of $G$.
			\item $\ell(J_G)=|E(G)|$ if and only if $\omega(G) \leq 3$. 
		\end{enumerate} 
	\end{corollary}
	\begin{proof}
		The assertion follows from the fact that $\displaystyle\F(J_G) = \F(J_{G_1}) \otimes_{\KK} \cdots \otimes_{\KK} \F(J_{G_c})$, and by Theorem \ref{normal}.
	\end{proof}
	
	We now obtain an upper bound for the regularity of special fiber ring of binomial edge ideals of closed graphs. Let $ \{f_1,\ldots,f_m\}$ be a minimal homogeneous generating set of an equi-generated homogeneous ideal $I$, and $Q=\KK[t_1,\ldots,t_m]$ be a standard graded polynomial ring over $\KK$.
	Let $\phi : Q \to S$ be the $\KK$-algebra homomorphism given by
	$\phi(t_{i}) = f_{i}$ for all  $i$. Then, $Q/\ker (\phi) \simeq$  Im$(\phi) = \F(I)$, where $\ker (\phi)$ is a homogeneous ideal of $Q$  and it is 
	called the \textit{defining ideal} of $\F(I)$. 
	
	\begin{notation}
		Let $I$ be an equi-generated homogeneous ideal in $S$. Then, by $\reg(\F(I))$, we mean the regularity of $Q/\ker(\phi)$ as $Q$-graded module. 
	\end{notation}

	\begin{theorem}\label{fiber-reg}
		Let $G$ be a closed graph on $[n]$ with respect to given labelling of vertices.  Assume that $G$ has no isolated vertices.  Then, $\reg(\F(J_G)) =\reg(\F(\ini_{\sigma}(J_G)))\leq  n-2c$, where $c$ is the number of components of $G$. 
	\end{theorem}
	\begin{proof}
		First, we assume that $G$ is a connected graph. It follows from the proof of Theorem \ref{normal} that     $\F(J_G)$ and $\F(\ini_{\sigma}(J_G))$ are  Cohen-Macaulay, and $\ell(\F(J_G))= \ell(\F(\ini_{\sigma}(J_G)))$. By \cite[Theorem 2.10]{ERT}, $\ini_{\sigma}(\F(J_G))=\F(\ini_{\sigma}(J_G)).$ Now, it follows from \cite[Proposition 2.4]{CHV} that $$\HS_{\F(J_G)}(\lambda) =\HS_{\ini_{\sigma}(\F(J_G))}(\lambda)=\HS_{\F(\ini_{\sigma}(J_G))}(\lambda),$$ and therefore, $$h_{\F(J_G)}(\lambda)=h_{\F(\ini_{\sigma}(J_G))}(\lambda).$$   Thus, $\reg(\F(J_G)) =\deg h_{\F(J_G)}(\lambda)=\deg h_{\F(\ini_{\sigma}(J_G))}(\lambda)=\reg(\F(\ini_{\sigma}(J_G)))$. 
		
		Let $H$ be the graph constructed in the proof of Theorem \ref{rees-reg}, i.e. $I(H)=\ini_{\sigma}(J_G)$. Now, it follows from \cite[Theorem 1]{HH202} that $\reg(\F(I(H))) \leq \text{mat}(H)-1=n-2$, as $\{ \{ x_i,y _{i+1}\} : 1 \leq i \leq n-1 \}$ is a perfect matching of $H$. Hence, $\reg(\F(J_G)) \leq  n-2$.
		
		Now, let $G=G_1\sqcup \cdots \sqcup G_c$. Then, $\displaystyle\F(J_G) = \F(J_{G_1}) \otimes_{\KK} \cdots \otimes_{\KK} \F(J_{G_c})$, and $\displaystyle\F(\ini_{\sigma}(J_G)) = \F(\ini_{\sigma}(J_{G_1})) \otimes_{\KK} \cdots \otimes_{\KK} \F(\ini_{\sigma}(J_{G_c}))$ which imples that    $\displaystyle\reg(\F(J_G))=\sum_{i=1}^c \reg(\F(J_{G_i})) =\sum_{i=1}^c \reg(\F(\ini_{\sigma}(J_{G_i})))=\reg(\F(\ini_{\sigma}(J_G))).$ Thus, $\reg(\F(J_G))=\reg(\ini_{\sigma}(J_G)) \leq \sum_{i=1}^c (|V(G_i)|-2)=n-2c.$ Hence, the assertion follows.   
	\end{proof}
	As an immediate consequence, we obtain the following:
	\begin{corollary}\label{reg-fiber}
		Let $G$ be a closed graph on  $[n]$. Assume that $G$ has no isolated vertices. Then, $\reg(\F(J_G)) \leq n-2c$, where $c$ is the number of connected components of $G$.
	\end{corollary}

	Now, we obtain  a lower bound for the regularity of special fiber ring of binomial edge ideals.
\begin{theorem}
	Let $G$ be a graph on $[n]$. If $H$ is an induced subgraph of $G$, then $$\reg(\F(J_H)) \leq \reg(\F(J_G)).$$ In particular, $\reg(\F(J_G))\geq \omega(G) -2 $ if $\omega(G) \geq 4$.
\end{theorem}
\begin{proof}
	Let $H$ be an induced subgraph of $G$. Clearly, $\F(J_H)$ is a $\KK$-subalgebra of $\F(J_G)$.  Set $R_1=\KK[T_{\{i,j\}} : \{i,j\} \in E(H) \text{ with } i < j]$ and $R_2=\KK[T_{\{i,j\}} : \{i,j\} \in E(G) \text{ with } i < j]$. Let $I_1$ and $I_2$ be ideals of $R_1$ and $R_2$, respectively such that $R_1/I_1 \simeq \F(J_H)$, and $R_2/I_2 \simeq \F(J_G)$. Now, define $\pi : R_2/I_2 \rightarrow R_1/I_1$ as  $\pi(\overline{T}_{\{i,j\}})=0$ if $\{i,j\} \not\in E(H)$, and  $\pi(\overline{T}_{\{i,j\}})=\overline{T}_{\{i,j\}}$ if $\{i,j\} \in E(H)$.  Consider, $R_1/I_1 \stackrel{i}\hookrightarrow R_2/I_2 \xrightarrow{\pi} R_1/I_1$. Then, $\pi \circ i$ is identity on $R_1/I_1$, and hence, $R_1/I_1$ is an algebra retract of $R_2/I_2$. It follows from \cite[Corollary 2.5]{HHH} that $\reg(R_1/I_1) \leq \reg(R_2/I_2)$. Hence, $\reg(\F(J_H) ) \leq \reg(\F(J_G))$. 
	
	Assume that $\omega(G) \geq 4 $. Let $H$ be a clique of $G$ such that $\omega(G)=|V(H)|$.  Then, $H$ is an induced subgraph of $G$.  Since $H$ is a complete graph, $H$ is closed with respect to the given labelling of vertices.  By Theorem \ref{fiber-reg}, $\reg(\F(J_H)) =\reg(\F(\ini_{\sigma}(J_H)))$. Now, it follows from \cite[Proposition 5.7]{AU09} that $\reg(\F(\ini_{\sigma}(J_H))) =|V(H)|-2=\omega(G)-2$.  Hence, $\reg(\F(J_G)) \geq \omega(G)-2$.
\end{proof}

	\vskip 2mm
	\noindent
	\textbf{Acknowledgements:}  The author sincerely  thanks Prof. Viviana Ene for valuable comments and suggestions. The author also wishes to express his sincere gratitude to the anonymous referees for useful comments.


\begin{thebibliography}{10}

\bibitem{dav}
Davide Bolognini, Antonio Macchia, and Francesco Strazzanti.
\newblock Binomial edge ideals of bipartite graphs.
\newblock {\em European J. Combin.}, 70:1--25, 2018.

\bibitem{dav2}
Davide {Bolognini}, Antonio {Macchia}, and Francesco {Strazzanti}.
\newblock {Cohen-Macaulay binomial edge ideals and accessible graphs}.
\newblock {\em arXiv e-prints}, page arXiv:2101.03619, January 2021.

\bibitem{Cid}
Yairon Cid-Ruiz.
\newblock Regularity and {G}r\"{o}bner bases of the {R}ees algebra of edge
  ideals of bipartite graphs.
\newblock {\em Matematiche (Catania)}, 73(2):279--296, 2018.

\bibitem{CHV}
Aldo Conca, J\"{u}rgen Herzog, and Giuseppe Valla.
\newblock Sagbi bases with applications to blow-up algebras.
\newblock {\em J. Reine Angew. Math.}, 474:113--138, 1996.

\bibitem{AU09}
Alberto Corso and Uwe Nagel.
\newblock Monomial and toric ideals associated to {F}errers graphs.
\newblock {\em Trans. Amer. Math. Soc.}, 361(3):1371--1395, 2009.

\bibitem{MG}
Marilena Crupi and Giancarlo Rinaldo.
\newblock Binomial edge ideals with quadratic {G}r\"{o}bner bases.
\newblock {\em Electron. J. Combin.}, 18(1):Paper 211, 13, 2011.

\bibitem{Ene}
Viviana Ene.
\newblock Syzygies of {H}ibi rings.
\newblock {\em Acta Math. Vietnam.}, 40(3):403--446, 2015.

\bibitem{her1}
Viviana Ene, J\"urgen Herzog, and Takayuki Hibi.
\newblock Cohen-{M}acaulay binomial edge ideals.
\newblock {\em Nagoya Math. J.}, 204:57--68, 2011.

\bibitem{ERT}
Viviana Ene, Giancarlo Rinaldo, and Naoki Terai.
\newblock Powers of binomial edge ideals with quadratic gr\"obner bases.
\newblock {\em Nagoya Math. J.}, To appear, 2021.

\bibitem{EZ}
Viviana Ene and Andrei Zarojanu.
\newblock On the regularity of binomial edge ideals.
\newblock {\em Math. Nachr.}, 288(1):19--24, 2015.

\bibitem{HH202}
J\"{u}rgen Herzog and Takayuki Hibi.
\newblock The regularity of edge rings and matching numbers.
\newblock {\em Mathematics}, 8(1):39, 2020.

\bibitem{HH1}
J\"urgen Herzog, Takayuki Hibi, Freyja Hreinsd\'ottir, Thomas Kahle, and
  Johannes Rauh.
\newblock Binomial edge ideals and conditional independence statements.
\newblock {\em Adv. in Appl. Math.}, 45(3):317--333, 2010.

\bibitem{her2}
J\"{u}rgen Herzog and Giancarlo Rinaldo.
\newblock On the extremal {B}etti numbers of binomial edge ideals of block
  graphs.
\newblock {\em Electron. J. Combin.}, 25(1):Paper 1.63, 10, 2018.

\bibitem{JACM}
A.~V. Jayanthan and Arvind Kumar.
\newblock Regularity of binomial edge ideals of {C}ohen-{M}acaulay bipartite
  graphs.
\newblock {\em Comm. Algebra}, 47(11):4797--4805, 2019.

\bibitem{JAR1}
A.~V. Jayanthan, Arvind Kumar, and Rajib Sarkar.
\newblock Regularity of powers of quadratic sequences with applications to
  binomial ideals.
\newblock {\em J. Algebra}, 564:98--118, 2020.

\bibitem{JAR}
A.~V. Jayanthan, Arvind Kumar, and Rajib Sarkar.
\newblock Almost complete intersection binomial edge ideals and their {R}ees
  algebras.
\newblock {\em J. Pure Appl. Algebra}, 225(6):106628, 2021.

\bibitem{AR2}
Arvind Kumar.
\newblock Binomial edge ideals of generalized block graphs.
\newblock {\em Internat. J. Algebra Comput.}, 30(8):1537--1554, 2020.

\bibitem{AR3}
Arvind kumar.
\newblock Binomial edge ideals and bounds for their regularity.
\newblock {\em J. Algebraic Combin.}, 53(3):729--742, 2021.

\bibitem{AR4}
Arvind Kumar.
\newblock Lov\'{a}sz-{S}aks-{S}chrijver ideals and parity binomial edge ideals
  of graphs.
\newblock {\em European J. Combin.}, 93:103274, 2021.

\bibitem{MM}
Kazunori Matsuda and Satoshi Murai.
\newblock Regularity bounds for binomial edge ideals.
\newblock {\em J. Commut. Algebra}, 5(1):141--149, 2013.

\bibitem{HHH}
Hidefumi Ohsugi, J\"{u}rgen Herzog, and Takayuki Hibi.
\newblock Combinatorial pure subrings.
\newblock {\em Osaka J. Math.}, 37(3):745--757, 2000.

\bibitem{OH2}
Hidefumi Ohsugi and Takayuki Hibi.
\newblock Koszul bipartite graphs.
\newblock {\em Adv. in Appl. Math.}, 22(1):25--28, 1999.

\bibitem{OH1}
Hidefumi Ohsugi and Takayuki Hibi.
\newblock Toric ideals generated by quadratic binomials.
\newblock {\em J. Algebra}, 218(2):509--527, 1999.

\bibitem{oh}
Masahiro Ohtani.
\newblock Graphs and ideals generated by some 2-minors.
\newblock {\em Comm. Algebra}, 39(3):905--917, 2011.

\bibitem{Rauf}
Asia Rauf and Giancarlo Rinaldo.
\newblock Construction of {C}ohen-{M}acaulay binomial edge ideals.
\newblock {\em Comm. Algebra}, 42(1):238--252, 2014.

\bibitem{MMK1}
Mohammad Rouzbahani~Malayeri, Sara Saeedi~Madani, and Dariush Kiani.
\newblock A proof for a conjecture on the regularity of binomial edge ideals.
\newblock {\em J. Combin. Theory Ser. A}, 180:Paper No. 105432, 9, 2021.

\bibitem{SVV}
Aron Simis, Wolmer~V. Vasconcelos, and Rafael~H. Villarreal.
\newblock On the ideal theory of graphs.
\newblock {\em J. Algebra}, 167(2):389--416, 1994.

\bibitem{Vill95}
Rafael~H. Villarreal.
\newblock Rees algebras of edge ideals.
\newblock {\em Comm. Algebra}, 23(9):3513--3524, 1995.

\end{thebibliography}
\end{document}